\documentclass{elsarticle}

\usepackage{amsmath,version}
by-\headheight \advance \topmargin by-\headsep 
\usepackage{amsmath,amssymb,amsthm,amsfonts}
\usepackage{graphicx}
\usepackage{color}
\usepackage{eepic}
\usepackage{epic}
\usepackage{pstricks}
\newtheorem{thm}{Theorem}[section]

\newtheorem{lem}[thm]{Lemma}

\theoremstyle{definition}

\newtheorem{rem}{Remark}[section]

\numberwithin{equation}{section}
\usepackage{appendix}

\usepackage{changes}

\DeclareMathSymbol{\C}{\mathalpha}{AMSb}{"43}
\DeclareMathSymbol{\R}{\mathalpha}{AMSb}{"52}

\newcommand{\ds}{\displaystyle}
\newcommand{\al}{\alpha}
\newcommand{\be}{\beta}
\newcommand{\de}{\delta}

\newcommand{\Om}{\Omega}
\newcommand{\si}{\sigma}
\newcommand{\ve}{\varepsilon}

\newcommand{\oo}{\infty}

\newcommand{\pa}{\partial}

\newcommand{\di}{\nabla\cdot}
\newcommand{\cu}{\nabla\times}

\newcommand{\bs}{\boldsymbol}
\newcommand{\non}{\nonumber}
\newcommand{\dsf}{\displaystyle \frac}
\newcommand{\ii}{\displaystyle\int_0^T}
\newcommand{\II}{\displaystyle\int_{\Omega}}

\newcommand{\AAA}{\boldsymbol a}
\newcommand{\BB}{\boldsymbol b}
\newcommand{\CC}{\boldsymbol c}

\newcommand{\nn}{\mathbf n}

\newcommand{\uu}{\boldsymbol u}
\newcommand{\vvv}{\boldsymbol v}

\newcommand{\xx}{\boldsymbol x}

\newcommand{\VN}{{\boldsymbol V}_{\boldsymbol N}}

\renewcommand{\blue}{\color{black}}

\def \er {{\bs e}_r}
\def \et {{\bs e}_\theta}
\def \ep {{\bs e}_\varphi}
\def \spht {{\bs T}_l^m}
\def \sphv {{\bs V}_l^m}
\def \sphw {{\bs W}_l^m}

\def \bbt {\bar{\beta}}
\title{An efficient numerical scheme  for a 3D spherical dynamo equation
}
\author[TC]{Ting Cheng\fnref{CT}}
\address[TC]{School of Mathematics and Statistics $\&$ Hubei Key Laboratory of Mathematical Sciences, Central China Normal University, Wuhan, 430079, P.R.CHINA}
{\fntext[CT]{Supported by NSF of China DOS 11871240  and DOS 11771170. }}
\author[LM]{Lina Ma\fnref{ML}}
\address[LM]{Department of Mathematics, Trinity College, Hartford, CT 06106, USA}
\fntext[ML]{Partially supported by NSF DMS-1913229.}
 \author[JS]{Jie Shen\fnref{SJ}}
 \address[JS]{Department of Mathematics, Purdue University, West Lafayette, IN 47907, USA. }
 \fntext[SJ]{Partially supported by NSF DMS-1620262, DMS-1720442 and AFOSR FA9550-16-1-0102.}

\begin{document}
\begin{abstract}
We develop an efficient numerical scheme for the 3D mean-field spherical dynamo equation. The scheme is based on a semi-implicit discretization in time and a spectral method in space based on the divergence-free spherical harmonic functions. A special  semi-implicit approach is proposed such that at each time step one only needs to solve a linear system  with constant coefficients. Then, using expansion in divergence-free spherical harmonic functions in the transverse directions allows us to reduce the linear system at each time step to a sequence of one-dimensional equations in the radial direction, which can then be efficiently solved by using a  spectral-element method. We show that the solution of   fully discretized scheme remains bounded independent of the number of unknowns, and present 
numerical results to validate our scheme.


\end{abstract}
\maketitle

\section{Introduction}

It is well known that many astrophysical bodies have intrinsic
magnetic fields. For examples,   Earth possesses a magnetic
field that has been known for many centuries; sunspots is   the
best-known manifestation of the solar magnetic activity cycle. But only in the last few decades scientists began to try to
understand more about the origin of these magnetic fields. It is widely
accepted that the magnetic activities of many planets and stars
represent the magnetohydrodynamic dynamo processes taking place in
their deep interiors. For the physical background of the dynamo
model, we refer to R. Hollerbach \cite{hollerbach1996theory} or
Chris A. Jones \cite{MR2768025} and the references therein.

There are numerous simplified mathematical models and numerical
simulations in the literature (see, e.g. Bullard et al. \cite{bullard1954homogeneous}, R. Hollerbach\cite{hollerbach2000spectral}, R. A. Bayliss, et al. \cite{MR2354034},
C. Guervilly, and P. Cardin,
\cite{MR2774729}, Chris A. Jones \cite{MR2768025}, W. Kuang and J. Bloxham
\cite{kuang1999numerical}, David
Moss \cite{MR2215943}, K. Zhang and F. Buss \cite{zhang1989convection} Paul H. Roberts, et al., \cite{MR2774728} and the
references therein).  There are also a few  studies with numerical analysis on some numerical methods for these
models, e.g.,{\blue
\cite{Chan06}, \cite{MR2537984}, \cite{zhang2003three}, \cite{chan2001nonlinear} and \cite{MR2602741}.} In
\cite{Chan06}, Chan, Zhang and Zou studied the
mathematical theory and its numerical approximation based on a
finite element method, while Mohammad M. Rahman and David R. Fearn
\cite{MR2602741} developed  a spectral approximation of some nonlinear mean-field dynamo
equations with different geometries and {\blue toroidal and poloidal decomposition}. 

There are two main difficulties in dealing with dynamo models: (i) it consists of  three-dimensional vector equations in spherical shells; and (ii) the magnetic field is implicitly divergence-free. Using a finite-element method to deal with the above issues may be complicated and costly. We  consider  in this paper the  model used in
\cite{Chan06} and propose an efficient numerical scheme based on a semi-implicit discretization in time and a spectral method in space based on the divergence-free spherical harmonic functions. we first discretize the model in time using a semi-implicit approach such that at each time step one only needs to solve a linear system  with piecewise constant coefficients. Then, we discretize  this linear system by using a spectral discretization consisting of divergence-free spherical harmonic functions in the transverse directions and a spectral-element method in the radial direction. This way, the linear system can be reduced to a sequence of one-dimensional equations in the radial direction for the coefficients of the expansion in divergence-free spherical harmonic functions so that it can be efficiently and accurately solved by using a 
spectral-element method.

The remainder of this paper is organized as follows. In  section 2, we describe the model that we consider,  list some of its mathematical properties, and some  useful mathematical tools that will be used later.
In Section 3, a fully
discrete spectral method for approximating the continuous
problem is proposed. The stability and the convergence analysis of
our numerical solutions are carried out in section 4. Section 5 contains implementation details, and a numerical experiment is shown in section 6 that demonstrates
the efficiency of our numerical scheme.

\section{Preliminaries}
\subsection{The Model}

We consider the
 following nonlinear spherical mean-field
dynamo system:
\begin{equation}
  \label{1}
 \left\{
  \begin{array}{lll}
  & \BB_t +\cu(\beta(\xx)\cu\BB)
  =R_\al \cu(\frac{f(x,t)}{1+\si|\BB|^2}\BB)+R_m \cu(\uu\times \BB)
     & \mbox{in} \ \ \Om \times (0,T),\\
     & \nabla\times\BB\times \nn=0
     & \mbox{on} \ \
      {\pa \Omega \times (0,T)},\\
    &\BB(\xx,0)=\BB^0(\xx)
     & \mbox{in} \ \ \Om.
  \end{array}
 \right.
 %\hspace{3.5cm}
\end{equation}
\begin{figure}[htp]
%\setlength{\unitlength}{.5cm}
%\begin{picture}(10, 14)
%\put(14,7){\circle{4}}
%\put(14,7){\circle{7}}
%\put(14,7){\circle{11}}
%\put(13.7,7){$\Omega_1$}
%\put(16.3,7){$\Omega_2$}
%\put(18,7){$\Omega_3$}
%\put(15.5,5){$\Gamma_1$}
%\put(16.5,4){$\Gamma_2$}
%%\put(13,0){$\rm {Figure} \ 1$}
%\end{picture}
\centering
\includegraphics[scale=.4]{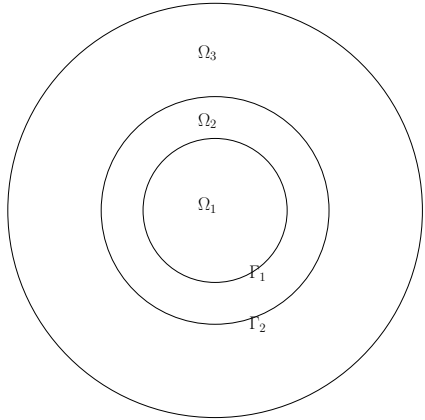}
\caption{Domain $\Omega$}\label{fig0}
\end{figure}

The unknown is the magnetic field $\BB$. $\Omega$ is the
physical domain of interest, which consists of three non-overlapping
zones $\Om_k\, (k=1,2,3)$ in spherical geometry (see Figure \ref{fig0}), where  $\Omega_1$ is the core, $\Omega_2$ is the convection zone and $\Omega_3$ is the outer photosphere.  $\nn$ denotes the unit outer normal vector to
the boundary of $\Om$.
The physical meanings
of the variables in \eqref{1} are as follows: $\uu=\uu(\xx,t)$
represents the fluid velocity field, which is given here, and
$f(\xx,t)$ is also a known function. Both $\uu$ and $f$ vanish on $\bar\Om_1$ and $\bar\Om_3$.
The non-dimensional parameters
$R_\al, \ R_m$ are Rayleigh numbers, $\si$ is a constant, $\beta(\xx)$ is
the magnetic diffusivity satisfying $\be_1\le \beta(\xx)\le \be_2$. The diffusivity is considered as constant in the convection zone. 
 At the two interfaces  $\Gamma_1$ and
$\Gamma_2$,  we impose  the physical jump conditions
\begin{equation}\label{1.1}
 [\be(\xx)\cu\BB\times \nn ]=0, \ [\BB ]=0
     \ \ \ \ \mbox{on} \ (\Gamma_1 \cup \Gamma_2) \times (0,T),
\end{equation}
where
 $[\AAA]$ denotes the jumps
of $\AAA$ across the interfaces and $\nn$ is the outward normal.

\vspace{0.5cm}

\begin{rem} Taking the divergence of the first equation in \eqref{1}, we find $\text{div} {\BB}_t=0$. Hence, if we impose the condition $\di \BB^0=0$, we have
  $\di \BB=0 \ \mbox{in} \ \ \Om \times (0,T)$.
\end{rem}

We now describe some notations, and recall some basic mathematical
properties for \eqref{1}.
We  denote by $H^m(\Om)\, (m \in
\mathbb{R})$ the usual Sobolev space, and  denote
$H^m(\Om)^3$ by $\boldsymbol{H}^m(\Om)$. As usual, $(\cdot, \cdot)$
denotes the scalar product in $\boldsymbol{L}^2(\Om)$ or
$L^2(\Om)$. For real $s \geq 0$, $\|\cdot\|_{s}$ denotes the norm
of $\boldsymbol{H}^s(\Om)$ (or the $H^s(\Om)$ for scalar
functions), in particular, we denote $\|\cdot\|_{0} \triangleq
\|\cdot\|.$
We define
$$
 {\bs V}=\{\CC \in \bs{L}^2(\Om);
   \mathbf{curl} \ \CC \in \bs{L}^2(\Om)\},$$
and for all $ \ \CC \in {\bs V},$ we set
$$\|\CC\|_{\bs V}^2=\|\CC\|^2+\|\cu \CC\|^2.$$

We   consider the following weak formulation for \eqref{1}:

Find $\BB(t) \in {\bs V}$ such that $\BB(0)=\BB_0$ and for almost all $t \in (0,T),$
\begin{eqnarray}\label{2}
 &&(\BB'(t),\AAA)+(\be \cu\BB(t),\cu\AAA)\non\\
    & = & R_\al \left(\frac{f(t)}{1+\si|\BB|^2}\BB(t), \cu \AAA\right)
     + R_m (\uu(t) \times \BB(t), \cu \AAA),
     \hspace{.8cm} \forall \ \AAA \in {\bs V}.
\end{eqnarray}

By using a standard argument (cf. M. Sermange and R. Temam
\cite{MR716200}), one can easily derive the following result:

\begin{thm} \label{lem1.1}
There exists a unique solution $\BB$ to the dynamo system \eqref{2} such that 
$$\BB \in L^\oo(0,T; \bs V) \cap H^1(0,T;\bs{L}^2(\Om))$$
provided that
$\BB_0 \in {\bs V}, \ f \in H^1(0,T;L^\oo(\Om)),
\ \uu \in H^1(0,T;\bs{L}^\oo(\Om))$. More precisely, there exists a constant $C>0$ such that
\begin{eqnarray} \label{5}
 &&\|\BB\|^2_{L^\oo(0,T;\bs V)}
    + \|\BB\|^2_{H^1(0,T;\bs{L}^2(\Om))}\non\\
 &\leq& C \left( \|\cu \BB^0\|^2 +\|\BB^0\|^2 \right)
         \max\limits_{0 \le t \le T}\left(\|f(t)\|_{L^\oo(\Om)}^2 + \|\uu(t)\|_{\bs L^\oo(\Om)}^2\right)\non\\
 && \cdot \exp \left(C \ii {\|f(t)\|_{L^\oo(\Om)}^2
     +\|f'(t)\|_{L^\oo(\Om)}^2+\|\uu(t)\|_{\bs L^\oo(\Om)}^2+\|\uu'(t)\|_{\bs L^\oo(\Om)}^2}\right).
\end{eqnarray}
\end{thm}

\subsection{Some Useful Mathematical Tools}
We recall below  some lemmas which will be used later.

\begin{lem}\label{lem1}
  (Young's inequality) For any $a, b \in \mathbb{R}$ and $\ve >0$, we have
  $$ab \le \ve a^2 + \frac{1}{4 \ve} b^2.$$
\end{lem}

\begin{lem}\label{lem2}
  (Discrete integration by parts) Let $\{\bs{a}_n\}_{n=0}^k$ and
  $\{\bs{b}_{ n}\}_{n=1}^k$ be two vector sequences, then we have
  $$\sum\limits_{n=1}^k (\bs{a}_{n}-\bs{a}_{n-1}) \cdot \bs{b}_n
    =\bs{a}_k \cdot \bs{b}_k -\bs{a}_0 \cdot \bs{b}_1
     -\sum\limits_{n=1}^{k-1} \bs{a}_{n} \cdot (\bs{b}_{n+1}-\bs{b}_{n}).$$
\end{lem}
\begin{proof} By direct calculation, we easily get
$$\sum\limits_{n=1}^k ({a}_n-{a}_{n-1}) \cdot {b}_n
    ={a}_k \cdot {b}_k -{a}_0 \cdot {b}_1
     -\sum\limits_{n=1}^{k-1} {a}_{n} \cdot ({b}_{n+1}-{b}_{n})$$
for scalar sequences $\{{a}_n\}_{n=0}^k$ and
$\{{b}_n\}_{n=1}^k$. The desired result for vector sequences  can be obtained
accordingly.
\end{proof}

\begin{lem}\label{lem3}
 (Gronwall inequality)
 Let $f \in L^1(t_0, T)$ be a non-negative function,
 $g$ and $\phi$
 be continuous functions on $[t_0, T]$. Moreover
 $g$ is non-decreasing. Then
 $$\phi(t) \le g(t)+\int_{t_0}^t f(\tau)\phi(\tau)\,d\tau
   \ \ \ \forall \ t \in [t_0, T]$$
 implies that
 $$\phi(t) \le g(t)e^{\int_{t_0}^t f(\tau)\,d\tau}
   \ \ \ \forall \ t \in [t_0, T].
 $$
\end{lem}

\begin{rem} We will frequently use the following special case:
 \begin{eqnarray}
  & f(t) \le C+\al\displaystyle\int_0^t f(s) d s \ \ \
   & \ \mbox{\rm implies that} \ \
    f(t) \le C e^{\al t} \ \ \ \forall \ t \in [0,T],
     \label{dc0.9.3}
\end{eqnarray}
where $\al \ge 0$ and $C$ are given constants.
\end{rem}

\begin{lem} \label{lem4}
  (\cite{MR851383}, p.34) (Integration by parts)
Let $\Om$ be a bounded region of $\mathbb{R}^d$ ($d=2$ or 3)
with Lipschitz continuous boundary. Then, 
the  mapping 
$$\gamma_\tau:
 \left\{
  \begin{array}{ll}
   \vvv \to \vvv \cdot \tau|_{\pa \Om} & \ {\rm for} \ \  d=2,\\
   \vvv \to \vvv \times \nn|_{\pa \Om} & \ {\rm for} \ \  d=3
  \end{array}
 \right.$$
can be extended by continuity to a linear and continuous mapping,
still denoted by $\gamma_\tau$, from $\bs V$ %$H(\mathbf{curl}; \Om)$
 into
$H^{-1/2}(\pa \Om)$ if $d=2$ or $H^{-1/2}(\pa \Om)^3$ if $d=3$,
where $\tau$ is the unit tangent vector to $\pa \Om$. Furthermore,
the following Green's formula holds:
\begin{equation}
 \label{l5}
  (\cu \vvv, \bs \phi)=(\vvv, \cu \bs \phi)
   -\langle\gamma_\tau \vvv, \bs \phi\rangle_{\pa \Om}
\end{equation}
$\forall \ \vvv \in \bs V,
 \forall \ \bs \phi \in H^1(\Om)^3 \ {\rm if} \ d=3
 \ {\rm or} \ \phi \in H^1(\Om) \ {\rm if} \ d=2.$
\end{lem}

\section{The Numerical Scheme}

%We construct in this section a numerical scheme consisting a
%semi-implicit discretization in time and a spectral discretization in
%space.

\subsection{Time Discretization}
 We consider uniform grid on the temporal scale $[0,T]$ with $\tau=\dsf{T}{K}$, and $t_i=i\tau$:
%First we divide the time interval $[0,T]$ into $M$ equally spaced
%subintervals using nodal points
\begin{equation}\label{2.1}
 0=t_0<t_1<\cdots<t_K=T
\end{equation}
Define $u^n=u(\cdot,t_{n})$ for $0 \le n
\le K$. For a given sequence $\{u^n\}_{n=0}^K \subset L^2(\Om)$, we
apply first order approximation via difference quotient and define the averaging term $\bar{u}^n$  as follows:
\begin{equation}\label{2.2}
 \pa_{\tau} u^n=\frac{u^n-u^{n-1}}{\tau}, \ \ \
 \bar{u}^n=\frac{1}{\tau}\int_{t_{n-1}}^{t_{n}} u(\cdot,t) dt, \quad 1\leq n\leq K,
\end{equation}
 and we set $\bar{u}^0=u(\cdot, 0).$

%For a given
%function $u$ defined on $\Om$, we denote $\ti u$ the extension of
%$u$ on $\ti \Om$.

In terms of  time discretization, we consider the following semi-implicit scheme.
For $n=1,2, \cdots, K$, find $\BB^n$ such that satisfies this differential equation
\begin{equation}\label{time1}
  \begin{split}
   \partial_{\tau} \BB^n
    &+ \cu(\bar\be\cu\BB^n)=
  \cu (\bar\be-\be(x) )\cu\BB^{n-1} \\
    &+R_\al \cu (\frac{\bar{f}^n}{1+\si|\BB_{\bs N}^{n-1}|^2}\BB^{n-1})
   +R_m \cu (\bar\uu^n \times \BB^{n-1}),
  \end{split}
\end{equation}
and the boundary conditions
\begin{equation}\label{time2}
\begin{split}
   & \nabla\times\BB^n\times \nn=0, \quad  \mbox{on}  \quad {\pa \Omega},\\
& [\bar \be\cu\BB^n\times \nn ]= [(\bar\be-\be(x))
    \cu\BB^{n-1}\times \nn ], \ [\BB^n ]=0
     \ \ \ \ \mbox{on} \;\; \Gamma_1 \cup \Gamma_2.
\end{split}
\end{equation}
where
$$\bar\be=\left\{
\begin{array}{ll}
 \bar\be_1\triangleq \max\limits_{x \in \Om_1} \be, \ \ & x \in \Om_1,\\
 \bar\be_2 \triangleq \max\limits_{x \in \Om_2} \be=\be,  & x \in \Om_2,\\
 \bar\be_3 \triangleq \max\limits_{x \in \Om_3} \be, \ \ & x \in \Om_3.
 \end{array}
\right.$$
\begin{rem}
Taking the divergence of \eqref{time1}, we find $\pa_\tau\nabla\cdot\BB^n=0$. Hence, $\nabla\cdot\BB^0=0$ implies $\nabla\cdot\BB^n=0$ for all $n\ge 1$. 
\end{rem}

\subsection{Spatial Discretization: Vector Spherical Harmonics(VSH)}

%{\bf Lina:} Add the dimension reduction, and the spectral element
%discretization, specify the space $\VN\subset V$, of the problem here...

For the spatial discretization, we are working with three dimensional variables in spherical region, therefore it is natural to consider basic functions that specifically designed for spherical domain. %we shall use the  VSH for the tangential directions and spectral-element method for the radial direction. 
%\subsubsection{Vector Spherical Harmonics(VSH)}
%We will recall the definition and some useful properties  of VSH.

Let $S$ be a unit sphere and $(r,\theta,\varphi)$ be the spherical coordinates with the moving (right-handed) coordinate basis
\begin{equation}\label{localco}
\begin{split}
&\er=
\big(\sin \theta \cos \varphi, \ \sin\theta \sin \varphi, \cos\theta\big),
\\ & \et=\big(\cos \theta \cos \varphi, \ \cos\theta \sin \varphi, -\sin\theta\big),\\
& \ep= \big(-\sin \varphi, \ \cos \varphi, 0\big).
\end{split}
\end{equation}
The tangential gradient is defined as
\begin{equation}\label{tgrad}
\nabla_S =\frac {\partial}{\partial\theta}\et+\frac 1 {\sin \theta}  \frac {\partial}{\partial\varphi}\ep.
\end{equation}
The spherical harmonic functions are defined via the associated Legendre polynomials:
$$
Y_l^m(\theta,\varphi)=\sqrt{\frac{2l+1}{4\pi}\frac{(l-m)!}{(l+m)!}}P_l^m(\cos \theta) e^{im\varphi}.
$$
Recall that   $\{Y_l^m\}$ form orthonormal basis functions of $L^2(S)$.
% and the set $\{\nabla_s Y_l^m \times \er,\nabla_s Y_l^m \}$ forms an orthogonal basis of the tangential vector space. 
Now we define the vector spherical harmonic functions (VSH) (see, e.g., \cite{Hill54,Nede01}), which form an orthogonal basis of ${\bs L}^2(S)$.
\begin{equation}\label{vsh}
\left\{
\begin{split}
&\spht=\nabla_S Y_l^m \times \er=\frac 1{\sin\theta}\frac{\partial Y_l^m}{\partial \varphi}\et-\frac{\partial Y_l^m}
{\partial \theta} \ep,\quad l\ge 1,\; |m|\le l,\\
&\sphv=(l+1)Y_l^m\er -\nabla_SY_l^m, \quad l\ge 0,\; |m|\le l,\\  
&\sphw=l Y_l^m \er+\nabla_S Y_l^m, \quad l\ge 1,\; |m|\le l,
\end{split}
\right.
\end{equation}
Some additional properties of VSH will be provided in Appendix $A$.

Given the above definitions, for any vector function $\bs F(\theta,\phi)$ defined on the sphere, we can decompose the function using VSH {\blue and some constant coefficients $\bar t_l^m, \bar v_l^m, \bar w_l^m$:} 
 \begin{equation}
{\blue
\bs F(\theta,\varphi)=\sum_{l=0}^\infty\sum_{|m|=0}^l
\Big[\bar t_{l}^m
\spht(\theta,\varphi) +
\bar v_{l}^m \sphv(\theta,\varphi)+\bar w_{l}^m \sphw(\theta,\varphi)\Big].
}
\end{equation}

Considering functions $\bs F(r,\theta, \varphi)$ defined in the three dimensional ball, since the radii direction and the tangential plane are perpendicular to each other, we  can decompose {\blue $\bs F$ using coefficient functions}:
 \begin{equation}
\bs F(r, \theta,\varphi)=\sum_{l=0}^\infty\sum_{|m|=0}^l
\Big[t_{l}^m(r)
\spht(\theta,\varphi) +
v_{l}^m(r) \sphv(\theta,\varphi)+w_{l}^m(r) \sphw(\theta,\varphi)\Big].
\end{equation}
\subsection{Solenoidal Vector Field}
One of the numerical challenge is how to maintain the divergence free property in the discrete case. In the traditionally methods, this usually involves staggered grid \cite{yee1966numerical}, Lagrange multiplier \cite{munz2000divergence} and penalty or projection methods \cite{karakashian1998nonconforming,baker1990piecewise}.

 %Given the special geometry, we can deal with this issue by using a divergence free (i.e., solenoidal) bases, which have been  used mostly in astrophysics \cite{bullard1954homogeneous}.  
{\blue  There exists a divergence free (i.e., solenoidal) basis, which have been used mostly in astrophysics \cite{bullard1954homogeneous}, that can take care of the divergence free condition automatically on the spherical domain. }
 Only till recently, there have been some research and analysis on this subject \cite{tuugluk2012direct} in the mathematical circle. The detailed derivation of the divergence free basis can be found in Appendix B. 
 
 We can expand any solenoidal vector function $\bs B(r,\theta,\varphi)$ as
 \begin{align}
\bs B(r,\theta,\varphi) = \sum_{l=0}^\infty\sum_{|m|=0}^l t_l^m(r) \spht (\theta, \varphi)+ \nabla \times (A_l^m(r) \spht (\theta, \varphi))+a_0^0(r) Y_0^0 {\bs e}_r.
\end{align}
with $\displaystyle d_2^{+} a_0^0(r) = 0$. The term $a_0^0(r) Y_0^0 {\bs e}_r$ will vanish once being applied to a curl operator, so in our problem, we will only consider $a_0^0(r)=0$.

\subsection{Weak formulation of full discretization}
{\blue We mark three intervals on the radial direction $I_1=[0,r_1], I_2=[r_1,r_2], I_3=[r_2,r_3]$, with $r_1, r_2, r_3$ be the radius of surfaces $\Gamma_1, \Gamma_2, \partial \Omega$ respectively. Each $I_i (i=1,2,3)$ is considered as an element on the radius. Let $\mathbb C_{N}$ be the complex polynomial space of degree at most $N$. We define the spectral-element space in the radial direction $X_{N}$ on $I=\{ I_1\cup I_2\cup I_3 \}$ by
\begin{align}
X_{N}=\{u_{N} |_{I_i} \in \mathbb C_{N}: [u]=0, {\rm i.e.}\; u_1(r_1)=u_2(r_1), u_2(r_2) = u_3(r_2)\}.
\end{align}
}
Let $Y_{M}$ be the truncated solenoidal vector field.
%\begin{align}
%Y_M=\text{span}\{ \spht,\sphv,\sphw: 1 \leq l \leq N_2, 0\leq|m|\leq l\}.
%\end{align}
We set ${\bs N}=(N,M)$. For a function $\BB_{\bs N}\in \VN:=X_N\times Y_M$, it can be expanded as
 \begin{align}
\BB_{\bs N}(r,\theta,\varphi) = \sum_{l=0}^M\sum_{|m|=0}^l t_{l,m}^N(r) \spht (\theta, \varphi)+ \nabla \times (A_{l,m}^N(r) \spht (\theta, \varphi)),
\end{align}
with $t_{l,m}^N(r), A_{l,m}^N(r) \in X_N$.

Then,
 our full discrete scheme is as follows: to find $\BB_{\bs N}^n\in \VN$
such that

%\begin{equation}\label{*}
% \begin{split}
%   \displaystyle \pa_\tau\BB_{\bs N}^n 
%    &+\displaystyle\bar\be \cu\cu\BB_{\bs N}^n
%   =\displaystyle \cu[ (\bar\be-\be(x)) (\cu\BB_{\bs N}^{n-1}) ]\\
%    &+R_\al \displaystyle \cu \frac{\bar{f}^n}{1+\si|\BB_{\bs N}^{n-1}|^2}\BB_{\bs N}^{n-1}
%       +R_m \displaystyle\cu (\bar\uu^n \times \BB_{\bs N}^{n-1})
%        \\
%  \end{split}
%\end{equation}
{\blue
\begin{equation}\label{*}
 \begin{split}
   \displaystyle&\int_{\Omega}\pa_\tau\BB_{\bs N}^n \cdot \AAA_{\bs N} d \bs x 
    +\displaystyle\int_{\Omega}\bar\beta (\cu\BB_{\bs N}^n) \cdot (\cu\AAA_{\bs N}) d \xx 
%    {-\int_{\partial \Omega}\bar\beta \cu\BB_{\bs N}^n\times {\bs n} \cdot \AAA_{\bs N} ds}\\
   =\displaystyle\int_{\Omega}(\bar\beta-\beta(x)) (\cu\BB_{\bs N}^{n-1}) \cdot (\cu\AAA_{\bs N}) d \bs x\\
    &+R_\alpha \displaystyle\int_{\Omega}\frac{\bar{f}^n}{1+\sigma|\BB_{\bs N}^{n-1}|^2}\BB_{\bs N}^{n-1}
       \cdot (\cu \AAA_{\bs N}) d \bs x 
%       &-R_{\alpha} \int_{\partial \Omega}\frac{\bar{f}^n}{1+\sigma|\BB_{\bs N}^{n-1}|^2}\BB_{\bs N}^{n-1}\times {\bs n} \cdot \AAA_{\bs N} ds\\
   +R_m \displaystyle\int_{\Omega}(\bar{\bs u}^n \times \BB_{\bs N}^{n-1}) 
    \cdot (\cu \AAA_{\bs N}) d \bs x,
%    &-R_m \displaystyle\int_{\pa\Omega}(\bar{\bs u}^n \times \BB_{\bs N}^{n-1})\times{\bs n}
%    \cdot  \AAA_{\bs N}d s
     \quad \forall \ \AAA_{\bs N} \in \VN,\\
  \end{split}
\end{equation}
}
%\begin{equation}\label{*}
% \begin{split}
%   \displaystyle\int_{\Omega}\pa_\tau\BB_{\bs N}^n \cdot \AAA_{\bs N} d \bs x 
%    &+\displaystyle\int_{\Omega}\bar\beta (\cu\BB_{\bs N}^n) \cdot (\cu\AAA_{\bs N}) d \xx 
%    {-\int_{\partial \Omega}\bar\beta \cu\BB_{\bs N}^n\times {\bs n} \cdot \AAA_{\bs N} ds}\\
%   =&\displaystyle\int_{\Omega}(\bar\beta-\beta(x)) (\cu\BB_{\bs N}^{n-1}) \cdot (\cu\AAA_{\bs N}) d \bs x\\
%    &+R_\alpha \displaystyle\int_{\Omega}\frac{\bar{f}^n}{1+\sigma|\BB_{\bs N}^{n-1}|^2}\BB_{\bs N}^{n-1}
%       \cdot (\cu \AAA_{\bs N}) d \bs x \\
%       &-R_{\alpha} \int_{\partial \Omega}\frac{\bar{f}^n}{1+\sigma|\BB_{\bs N}^{n-1}|^2}\BB_{\bs N}^{n-1}\times {\bs n} \cdot \AAA_{\bs N} ds\\
%   &+R_m \displaystyle\int_{\Omega}(\bar{\bs u}^n \times \BB_{\bs N}^{n-1}) 
%    \cdot (\cu \AAA_{\bs N}) d \bs x\\
%    &-R_m \displaystyle\int_{\pa\Omega}(\bar{\bs u}^n \times \BB_{\bs N}^{n-1})\times{\bs n}
%    \cdot  \AAA_{\bs N}d s
%     \quad \forall \ \AAA_{\bs N} \in \VN,\\
%  \end{split}
%\end{equation}
with
\begin{align}\label{**}
\begin{split}
%&    \nabla \times \BB_{\bs N}^n \times \nn  = 0   \quad \mbox{on} \; \Gamma_3\\
%& [\bar \be\cu\BB_{\bs N}^n\times \nn ]=[(\bar \be-\be(x))\cu\BB_{\bs N}^{n-1}\times \nn ], 
\ [\BB_{\bs N}^n ]=0
     \ \ \ \ \mbox{on} \ \Gamma_1 \cup \Gamma_2
     \end{split}
\end{align}
for $n=1,2, \cdots, K,$ and with initial condition
\begin{equation}\label{2.9}
 \BB_{\bs N}^0 = \Pi_{\bs N} \BB_0(\xx),
\end{equation}
where $\Pi_{\bs N}$ is the projection into the solenoidal vector field.

%
%More precisely, \eqref{time1}-\eqref{time2} is equivalent to \eqref{*}-\eqref{**}. And if we set $\BB_N=\BB^n_N$ , we can rewrite the system as
%\begin{align}\label{varfmNM}
%\begin{split}
%&\alpha<\BB_{\bs N}, {\AAA_{\bs N}}> +\bar \beta <\cu \BB_{\bs N}, \cu {\AAA_{\bs N}}>-[\bar\beta <\cu \BB_{\bs N} \times \nn , {\AAA_{\bs N}}>_S] \\
%=& <{\bf f}_1,{\AAA_{\bs N}}>+<{\bf f}_2, \cu{\AAA_{\bs N}}> -[ <{ \bf f}_2 \times \nn, {\AAA_{\bs N}}>_S], \quad \forall \ \AAA_{\bs N} \in \VN,
%\end{split}
%\end{align}
%\begin{equation}\label{spt2}
%\begin{split}
%   & \nabla\times\BB\times \nn=0, \quad  \mbox{on} \quad   { \Gamma_3},\\
%& [\bar\beta\cu\BB\times \nn ]=[{\bf g}], \ [\BB]=0
%     \ \ \ \ \mbox{on} \ \Gamma_1 \cup \Gamma_2.
%\end{split}
%\end{equation}

%\begin{align}\label{wkt}
%  \begin{split}
%& (\bbt d_r t_l^m,d_r \phi^t_{l,m})_{r^2} +\alpha(t_l^m,\phi^t_{l,m})_{r^2}+ l(l+1)(\bbt t_l^m,\phi_{l,m}^t) \\
%&+ a(\bbt_1-\bbt_2)t_{l}^m(a)\phi^t_{l,m}(a) + b(\bbt_2-\bbt_3)t_{l}^m(b)\phi^t_{l,m}(b) \\%+ c\bbt_3 t_3(c)\phi_3(c)\\
%=&(f^{1,t}_{l,m}, \phi^t_{l,m})_{r^2} + (  \{l f^{2,v}_{l,m} + (l+1) f^{2,w}_{l,m} \}, \phi^t_{l,m})_r + (\{f^{2,v}_{l,m}- f^{2,w}_{l,m}\}, d_r\phi^t_{l,m})_{r^2}\\
%&+ a^2 g^{1,T}_{l,m}\phi^t_{l,m}(a) - b^2 g^{2,T}_{l,m}\phi^t_{l,m}(b).
%  \end{split}
%\end{align}

\section{Stability analysis}
We show in this section that the solution of the  fully discretized scheme remain bounded. 
%{\bf Remark.} $\di \BB_{\bs N} \neq 0$, but $\di \BB_{\bs N} \to 0$ as $N \to \oo$.

\begin{thm}\label{lem2.4}
{\blue Let $\BB_{\bs N}^n$ be the solution of the spectral
method \eqref{*}-\eqref{2.9}. We assume $f \in W^{1,\oo}(0,T;L^\oo(\Om))$ and $\uu \in
 W^{1,\oo}(0,T;\bs{L}^\oo(\Om))$.
% $(\ti\BB_h^n, \ti p_h^n)$ is the extension of $(\BB_h^n, p_h^n)$ over $\ti \Om$.
Then 
there exists positive constants $C$, % and $\delta$, 
 independent of $N$, such that the following inequalities hold.} %for any $\tau < \delta$, we have}
 \begin{eqnarray}
  && \ \ \max\limits_{1 \le n \le M} \|\BB_{\bs N}^n\|^2
   +\tau \sum\limits_{n=1}^M \|\cu \BB_{\bs N}^n\|^2
     \le C (\|\BB^0_{\bs N}\|_{\bs  V}^2 + \tau \|\cu \BB^0_{\bs N}\|^2), \label{thm1}\\
  && \ \ \max\limits_{1 \le n \le M} \|\cu \BB_{\bs N}^n\|^2
    +\tau \sum\limits_{n=1}^M \|\pa_\tau \BB_{\bs N}^n\|^2
     \le C \|\BB^0_{\bs N}\|_{\bs  V}^2. \label{thm2}
 \end{eqnarray}
\end{thm}
\begin{proof} 
Let $0\le \de=\max \dsf{\bar \be-\be}{\bar \be}<1$.
Taking $\AAA_{\bs N}= 2\tau \BB_{\bs N}^n$ in
  \eqref{*}, using the Cauchy-Schwarz inequality, Young inequality and the regularity assumption on $f$ and ${\bs u}$, we can derive
\begin{eqnarray*}
\|\BB_{\bs N}^n\|^2& -&\|\BB_{\bs N}^{n-1}\|^2 +\|\BB_{\bs N}^n -\BB_{\bs N}^{n-1}\|^2
  +2\tau \bar\beta \|\nabla \times \BB_{\bs N}^n\|^2 \\
&\leq & \tau \|\bar\beta \nabla \times \BB_{\bs N}^n\|^2
 +\frac{\tau}{\bar \beta} \int_{\Omega} (\bar \beta -\beta(x))^2
 |\nabla \times \BB_{\bs N}^{n-1}|^2 \,d \xx\\
 && +\frac{\epsilon}{2} \tau \|\bar
 \beta \nabla \times \BB_{\bs N}^n\|^2
  +\frac{2}{\epsilon \bar \beta} \tau R^2_{\alpha}
 \int_{\Omega} \left(\frac{\bar f^n}{1+ \sigma |\BB_{\bs N}^{n-1}|^2}\right)^2 |\BB_{\bs N}^{n-1}|^2
 \,d \xx\\
 &&+\frac{\epsilon}{2}\tau \|\bar \beta \nabla \times \BB_{\bs N}^n\|^2
 +\frac{2}{\epsilon \bar \beta} \tau R_m^2 \int_{\Omega} |\bar \uu^n \times
 \BB^{n-1}_{\bs N}|^2 \,d \xx\\
&\leq & \tau \|\bar \beta \nabla \times \BB_{\bs N}^n\|^2 +\tau
\delta^2\|\bar \beta \nabla \times \BB_{\bs N}^{n-1}\|^2 +\tau \epsilon
\|\bar \beta \nabla \times \BB_{\bs N}^{n}\|^2 +C\tau \|\BB_{\bs N}^{n-1}\|^2,
\end{eqnarray*}
which implies
%and
%$$\bar \be \|\BB\|^2 \triangleq \bar \be_{1} \|\BB\|_{\Om_1}^2
% + \bar \be_{2} \|\BB\|_{\Om_2}^2 + \bar \be_{3} \|\BB\|_{\Om_3}^2. $$
\[
\|\BB_{\bs N}^n\|^2 -\|\BB_{\bs N}^{n-1}\|^2 +\tau (1-\epsilon)
\|\bar\beta \nabla \times \BB_{\bs N}^n\|^2 \leq \tau \delta^2 \|\bar \beta \nabla
\times \BB_{\bs N}^{n-1}\|^2 +C\tau \|\BB_{\bs N}^{n-1}\|^2.
\]
Taking $\epsilon$ small enough such that $1-\epsilon \geq \delta^2
+\epsilon$, i.e., $\epsilon \leq \frac{1}{2} (1-\delta^2)$, we get
\[
\|\BB_{\bs N}^n\|^2 -\|\BB_{\bs N}^{n-1}\|^2 +\tau \epsilon
\|\bar\beta \nabla \times \BB_{\bs N}^n\|^2 + \tau \delta^2 (\|\bar \beta \nabla
\times \BB_{\bs N}^{n}\|^2 - \|\bar \beta \nabla \times \BB_{\bs N}^{n-1}\|^2) \le C\tau \|\BB_{\bs N}^{n-1}\|^2.
\]
Summing up the above relation for $n$ from $1$ to ${M}$, we arrive at
\[
\|\BB^M_{\bs N}\|^2 +\epsilon \tau \sum_{n=1}^M \| \bar\beta \nabla
\times \BB_{\bs N}^n\|^2 \leq \|\BB_{\bs N}^0\|^2 +\tau \delta^2
\|\bar \beta \nabla \times \BB_{\bs N}^0\|^2 +C\tau \sum_{n=0}^{M-1} \|\BB_{\bs N}^n\|^2,
\]
which can also be written as
%\[
%\|\BB_{\bs N}\|^2 +\tau \sum_{n=1}^m \| \nabla \times \BB_{\bs N}^n\|^2 \leq
%C(\|\BB_{\bs N}^0\|^2 +\tau \|\nabla \times \BB_{\bs N}^0\|^2) +C\tau
%\sum_{n=0}^{m-1} \|\BB_{\bs N}^n\|^2.
%\]
{
\[
\|\BB_{\bs N}^M\|^2 +\tau \sum_{n=1}^M \| \nabla \times \BB_{\bs N}^n\|^2 \leq
C(\|\BB_{\bs N}^0\|^2 +\tau \|\nabla \times \BB_{\bs N}^0\|^2) +C\tau
\sum_{n=0}^{M-1} \|\BB_{\bs N}^n\|^2.
\]
}
Applying the {\blue discrete Gronwall's inequality to above inequality}, we find
\[
\max_{1\leq n \leq M} \|\BB_{\bs N}^n\|^2 +\tau \sum_{n=1}^M \| \nabla
\times \BB_{\bs N}^n\|^2 \leq C(\|\BB_{\bs N}^0\|^2 +\tau \|\nabla \times
\BB_{\bs N}^0\|^2)\leq C(\|\BB^0_{\bs N}\|_{\bs  V}^2 +\tau \|\nabla \times \BB^0_{\bs N}\|^2),
\]
which is \eqref{thm1}.

To prove \eqref{thm2}, we take $\AAA_{\bs N}=\tau \pa_\tau
\BB_{\bs N}^n=\BB_{\bs N}^n-\BB_{\bs N}^{n-1}$ in \eqref{*} to
obtain
\begin{eqnarray*}
 \tau \|\pa_{\tau}\BB_{\bs N}^n\|^2
   &+&\displaystyle\int_{\Om}\bar\be (\cu\BB_{\bs N}^n) \cdot \cu (\BB_{\bs N}^n-\BB_{\bs N}^{n-1}) \, d \xx\\
 && -\displaystyle\int_{\Om}(\bar\be-\be) (\cu\BB_{\bs N}^{n-1}) \cdot \cu (\BB_{\bs N}^n-\BB_{\bs N}^{n-1}) \, d \xx\\
 &=& R_\al \displaystyle\int_{\Om}\frac{\bar{f}^n}{1+\si|\BB_{\bs N}^{n-1}|^2}\BB_{\bs N}^{n-1}
      \cdot \cu (\BB_{\bs N}^n-\BB_{\bs N}^{n-1}) \, d \xx\\
 && +R_m \displaystyle\int_{\Om}(\bar\uu^n \times \BB_{\bs N}^{n-1})
    \cdot \cu (\BB_{\bs N}^n-\BB_{\bs N}^{n-1}) \, d \xx.
\end{eqnarray*}
We derive from the above that
\begin{eqnarray*}
 \tau \|\pa_{\tau}\BB_{\bs N}^n\|^2
  &+&\|\bar\beta \nabla \times \BB_{\bs N}^n\|^2
   +\II (\bar \be-\be)|\cu \BB_{\bs N}^{n-1}|^2 \,d \xx\\
 &= & \II (2\bar \be-\be)(\cu \BB_{\bs N}^n ) \cdot (\cu \BB_{\bs N}^{n-1} ) \, d \xx\\
 && + R_\al \displaystyle\int_{\Om}\frac{\bar{f}^n}{1+\si|\BB_{\bs N}^{n-1}|^2}\BB_{\bs N}^{n-1}
      \cdot \cu (\BB_{\bs N}^n-\BB_{\bs N}^{n-1}) \, d \xx\\
 && +R_m \displaystyle\int_{\Om}(\bar\uu^n \times \BB_{\bs N}^{n-1})
    \cdot \cu (\BB_{\bs N}^n-\BB_{\bs N}^{n-1}) \, d \xx\\
 &\leq & \frac{1}{2} \II (2\bar \be-\be)|\cu \BB_{\bs N}^n|^2 d \xx
  +\frac{1}{2} \II (2\bar \be-\be)|\cu \BB_{\bs N}^{n-1}|^2 \, d \xx\\
 && + R_\al \displaystyle\int_{\Om}\frac{\bar{f}^n}{1+\si|\BB_{\bs N}^{n-1}|^2}\BB_{\bs N}^{n-1}
      \cdot \cu (\BB_{\bs N}^n-\BB_{\bs N}^{n-1}) \, d \xx \\
 && +R_m \displaystyle\int_{\Om}(\bar\uu^n \times \BB_{\bs N}^{n-1})
    \cdot \cu (\BB_{\bs N}^n-\BB_{\bs N}^{n-1}) \, d \xx,
\end{eqnarray*}
which can be rewritten as
\begin{eqnarray*}
 \tau \|\pa_{\tau}\BB_{\bs N}^n\|^2
   &+& \II \frac{\be}{2}(|\cu \BB_{\bs N}^{n}|^2-|\cu \BB_{\bs N}^{n-1}|^2) d \xx\\
  & \leq & R_\al \displaystyle\int_{\Om}\frac{\bar{f}^n}{1+\si|\BB_{\bs N}^{n-1}|^2}\BB_{\bs N}^{n-1}
      \cdot \cu (\BB_{\bs N}^n-\BB_{\bs N}^{n-1}) d \xx \\
 && +R_m \displaystyle\int_{\Om}(\bar\uu^n \times \BB_{\bs N}^{n-1})
    \cdot \cu (\BB_{\bs N}^n-\BB_{\bs N}^{n-1}) d \xx.
\end{eqnarray*}
Summing up the above for $n$ from $1$ to $M$ leads to
\begin{eqnarray}\label{*1}
  \tau \sum_{n=1}^M \|\pa_{\tau}\BB_{\bs N}^n\|^2
  & +&\frac{\be_1}{2}\|\cu \BB_{\bs N}^{M}\|^2\non\\
 & \leq & \frac{\be_2}{2}\|\cu \BB_{\bs N}^{0}\|^2
    +R_\al \sum_{n=1}^M \displaystyle\int_{\Om}\frac{\bar{f}^n}{1+\si|\BB_{\bs N}^{n-1}|^2}\BB_{\bs N}^{n-1}
      \cdot \cu (\BB_{\bs N}^n-\BB_{\bs N}^{n-1}) \, d \xx\non\\
 && +R_m \sum_{n=1}^M \displaystyle\int_{\Om}(\bar\uu^n \times \BB_{\bs N}^{n-1})
    \cdot \cu (\BB_{\bs N}^n-\BB_{\bs N}^{n-1}) \, d \xx\non\\
 &\triangleq& \frac{\be_2}{2}\|\cu \BB_{\bs N}^{0}\|^2+I+II.
\end{eqnarray}
Next, we estimate $I$ and $II$ as follows.

By discrete integration by parts (cf. Lemma \ref{lem2}), we have
\begin{eqnarray*}
\sum_{n=1}^M  \frac{\bar{f}^n}{1+\si|\BB_{\bs N}^{n-1}|^2}\BB_{\bs N}^{n-1}
      &\cdot & \cu (\BB_{\bs N}^n-\BB_{\bs N}^{n-1})\\
 &=& \frac{\bar{f}^M \BB_{\bs N}^{M-1}}{1+\si|\BB_{\bs N}^{M-1}|^2} \cdot \cu \BB_{\bs N}^M
   -\frac{\bar{f}^1 \BB_{\bs N}^{0}}{1+\si|\BB_{\bs N}^{0}|^2} \cdot \cu \BB_{\bs N}^0 \\
 && -\sum_{n=1}^{M-1}  \left(\frac{\bar{f}^{n+1}\BB_{\bs N}^n}{1+\si|\BB_{\bs N}^{n}|^2}
      -\frac{\bar{f}^{n}\BB_{\bs N}^{n-1}}{1+\si|\BB_{\bs N}^{n-1}|^2}\right)
       \cdot \cu \BB_{\bs N}^n.
\end{eqnarray*}
Hence, it is easy  to derive from the above that
\begin{eqnarray}\label{*2}
 |I|
 &\le& \dsf{\be_1}{8}\|\cu\BB_{\bs N}^M\|^2+C\|\BB_{\bs N}^0\|_{\bs  V}^2
  +\dsf{\tau}{8}\sum_{n=1}^M\|\pa_\tau \BB_{\bs N}^n\|^2\non\\
 && +\sum_{n=1}^{M-1} \displaystyle\int_{\Om}
    \frac{\left|\si\bar{f}^{n+1}|\BB_{\bs N}^{n-1}|^2\BB_{\bs N}^n-\si\bar{f}^{n}|\BB_{\bs N}^{n}|^2\BB_{\bs N}^{n-1}\right|}
         {(1+\si|\BB_{\bs N}^{n}|^2)(1+\si|\BB_{\bs N}^{n-1}|^2)}
      \cdot |\cu \BB_{\bs N}^n| \, d \xx\non\\
 &\triangleq& \dsf{\be_1}{8}\|\cu\BB_{\bs N}^M\|^2+C\|\BB_{\bs N}^0\|_{\bs  V}^2
  +\dsf{\tau}{8}\sum_{n=1}^m\|\pa_\tau \BB_{\bs N}^n\|^2
   +III.
\end{eqnarray}
Since $f \in W^{1,\oo}(0,T;L^\oo)$,
the term $III$ can be estimated as follows:
\begin{eqnarray*}
 III&=& \sum_{n=1}^{M-1} \displaystyle\int_{\Om}
    \frac{\Big|\si\bar{f}^{n+1}|\BB_{\bs N}^{n-1}|^2\BB_{\bs N}^n-\si\bar{f}^{n}|\BB_{\bs N}^{n}|^2\BB_{\bs N}^{n-1}\Big|}
         {(1+\si|\BB_{\bs N}^{n}|^2)(1+\si|\BB_{\bs N}^{n-1}|^2)}
      \cdot |\cu \BB_{\bs N}^n| \, d \xx\non\\
 &= & \si\sum_{n=1}^{M-1} \displaystyle\int_{\Om}
    \frac{\Big|\bar{f}^{n+1}(\BB_{\bs N}^{n}-\BB_{\bs N}^{n-1})|\BB_{\bs N}^{n-1}|^2+\bar{f}^{n+1}(|\BB_{\bs N}^{n-1}|^2-|\BB_{\bs N}^n|^2)\BB_{\bs N}^{n-1}
          +(\bar{f}^{n+1}-\bar{f}^{n})|\BB_{\bs N}^{n}|^2\BB_{\bs N}^{n-1}\Big|}
         {(1+\si|\BB_{\bs N}^{n}|^2)(1+\si|\BB_{\bs N}^{n-1}|^2)}
       |\cu \BB_{\bs N}^n| \, d \xx\non\\
 &\leq& C\tau \sum_{n=1}^{M-1} \displaystyle\int_{\Om}|\pa_{\tau} \BB_{\bs N}^n||\cu \BB_{\bs N}^n| \, d \xx 
        +C\sum_{n=1}^{M-1} \displaystyle\int_{\Om}
         \frac{\Big|\si\BB_{\bs N}^{n-1}(|\BB_{\bs N}^{n}|+|\BB_{\bs N}^{n-1}|)(|\BB_{\bs N}^{n}|-|\BB_{\bs N}^{n-1}|)\Big|}
          {(1+\si|\BB_{\bs N}^{n}|^2)(1+\si|\BB_{\bs N}^{n-1}|^2)}
       |\cu \BB_{\bs N}^n| \, d \xx \non\\
 && +C \tau \sum_{n=1}^{M-1}  \displaystyle\int_{\Om} |\BB_{\bs N}^{n-1}||\cu \BB_{\bs N}^{n}| \, d \xx,
 \end{eqnarray*}
which can be further estimated by 
\begin{eqnarray*}
 III&\leq & \frac{\tau}{8} \sum_{n=1}^{M-1}\|\pa_\tau \BB_{\bs N}^n\|^2+C \tau  \sum_{n=1}^{M-1} \|\cu\BB_{\bs N}^n\|^2
          +C \sum_{n=1}^{M-1} \displaystyle\int_{\Om}|\tau \pa_{\tau} \BB_{\bs N}^n|
             \frac{\si |\BB_{\bs N}^{n-1}|^2+\frac{\si}{2}(|\BB_{\bs N}^{n}|^2+|\BB_{\bs N}^{n-1}|^2)}{(1+\si|\BB_{\bs N}^{n}|^2)(1+\si|\BB_{\bs N}^{n-1}|^2)}|\cu \BB_{\bs N}^n| \, d \xx\non\\
 && +C \tau  \sum_{n=1}^{M-1} \|\cu\BB_{\bs N}^n\|^2+C \tau \sum_{n=1}^{M-1}\|\BB_{\bs N}^{n-1}\|^2\non\\
 &\leq & \frac{\tau}{8} \sum_{n=1}^{M-1}\|\pa_\tau \BB_{\bs N}^n\|^2+C \tau  \sum_{n=1}^{M-1} \|\cu\BB_{\bs N}^n\|^2
          +2C \sum_{n=1}^{m-1} \displaystyle\int_{\Om}|\tau \pa_{\tau} \BB_{\bs N}^n|
             |\cu \BB_{\bs N}^n| \, d \xx +C \tau \sum_{n=1}^{M-1}\|\BB_{\bs N}^{n-1}\|^2\non\\
 &\le& \dsf{\tau}{4} \sum_{n=1}^M \|\pa_{\tau} \BB_{\bs N}^n\|^2 +C\|\BB_{\bs N}^0\|_{\bs V}^2.
\end{eqnarray*}

Similarly, we can derive 
\begin{equation}\label{*4}
 |II|\le \dsf{\be_1}{8}\|\cu\BB_{\bs N}^M\|^2+C\|\BB_{\bs N}^0\|_{\bs  V}^2
  +\dsf{\tau}{8}\sum_{n=1}^M\|\pa_\tau \BB_{\bs N}^n\|^2.
\end{equation}

Therefore, we obtain from \eqref{*1}-\eqref{*4} that
$$
  \dsf{\tau}{2}\sum_{n=1}^M\|\pa_\tau \BB_{\bs N}^n\|^2
  +\dsf{\be_1}{4}\|\cu\BB_{\bs N}^M\|^2
  \le \dsf{\be_2}{2}\|\cu\BB_{\bs N}^0\|^2+C\|\BB_{\bs N}^0\|_{\bs  V}^2
  \le C\|\BB^0\|_{\bs  V}^2,
$$
which implies the desired result.
\end{proof}

{\blue Note that the above theorem only shows that the scheme is unconditionally stable. However, to obtain accurate approximations, one still needs to choose a time step, which should depend on physical parameters $R_m,\;R_{\alpha}$ and $\beta_{\xx}$,  sufficiently small so that the  dynamical behavior  can be corrected captured.
With the above stability result, one can follow a  standard, albeit tedious, procedure to derive
an error estimate by assuming further regularity on the solution. For the sake of brevity, we leave this to the interested reader.}

\section{Numerical Implementation}

We will describe the details in numerical implementation in this section. it is natural to apply a spectral element treatment to the expansion, to accommodate the phenomenon in three different domains. 
%For the magnetic field $\bs B(r,\theta,\phi)$
%\subsection{Detailed Numerical Scheme}
We present  the expansion in terms of  Heaviside step function $u_I$. 
\begin{align}\label{ExpanNM}
&\bs B(r,\theta,\varphi) = \sum_{i=1}^3\sum_{l=0}^\infty \sum_{|m|= 0}^{l} u_{I_i}\Big[  t_{i,l,m}(r) \spht(\theta,\varphi) + \nabla\times \Big( A_{i,l,m}(r) \spht(\theta,\varphi) \Big) \Big],
%&\bs A_{\bs N} (r, \theta, \varphi) = \sum_{l=0}^M \sum_{|m|=0}^{l} \Big [ \phi_{l,m}^t(r) \spht(\theta,\varphi) + \nabla \times \Big(\phi_{l,m}^A(r) \spht(\theta,\varphi) \Big)\Big],
\end{align}
where $\{ t_{i,l,m}, A_{i,l,m}\} \in \mathbb{C}_N(I_i)$. 

Under this expansion, one can find two fully decoupled systems for $t_{i,l,m}$ and $A_{i,l,m}$. And immediately the three dimension problem is reduced into a system of one dimension problems. 

We first define some notations for cleaner form. Denote: 
\begin{align}
& \alpha = \frac{1}{\tau},\quad {\bs f}_1 = \alpha \BB^{n-1},\\
&{\bs f}_2 = (\bar\beta-\beta(x) )\cu\BB^{n-1}
    +R_\alpha \frac{\bar{f}^n}{1+\sigma|\BB^{n-1}|^2}\BB^{n-1}
   +R_m (\bar{\bs u}^n \times \BB^{n-1}),\\
&{\bs g}= \left\{
\begin{array}{ll}
{\bs g}_1=(\bar\beta_1 - \beta_1(x)) \cu \BB^{n-1} \times \nn, & x\ \rm{on} \ \Gamma_1,\\
{\bs g}_2=-(\bar\beta_3 - \beta_3(x)) \cu \BB^{n-1} \times \nn, & x\ \rm{on} \ \Gamma_2.
\end{array}
\right.
\end{align}
Equation (\ref{time1}) can then be written in this form:
\begin{align}
\alpha \BB^n +\bar\beta \nabla\times\nabla \times \BB^n = \bs f_1+\nabla \times \bs f_2,
\end{align}
with the boundary conditions
\begin{equation}
\begin{split}
   & \nabla\times\BB^n\times \nn=0, \quad  \mbox{on}  \quad {\pa \Omega},\\
& [\bar \be\cu\BB^n\times \nn ]= \bs g, \ [\BB^n ]=0
     \ \ \ \ \mbox{on} \;\; \Gamma_1 \cup \Gamma_2.
\end{split}
\end{equation}
We apply the harmonic vector spherical analysis to $\bs f_1, \bs f_2$ and $\bs g$. It is clear that $\bs f_1$ is also in the solenoidal field, but it can still be expanded with the full dimension analysis. 
\begin{align}
&\bs f_1(r,\theta,\varphi) = \sum_{i=1}^3\sum_{l=0}^\infty \sum_{|m|= 0}^{l} u_{I_i}\Big[  f_{i,l,m}^{1,T} \spht + f_{i,l,m}^{1,\nabla_S} \nabla_S Y_l^m + f_{i,l,m}^{1,r} {\bs e}_r\Big],\\
&\bs f_2(r,\theta,\varphi) = \sum_{i=1}^3\sum_{l=0}^\infty \sum_{|m|= 0}^{l} u_{I_i}\Big[  f_{i,l,m}^{2,T} \spht + f_{i,l,m}^{2,\nabla_S} \nabla_S Y_l^m + f_{i,l,m}^{2,r} {\bs e}_r\Big],\\
&\bs g_i(\theta,\varphi)= \sum_{l=0}^\infty \sum_{|m|= 0}^{l} \Big[  g_{i,l,m}^{T}(r) \spht(\theta,\varphi) + g_{i,l,m}^{\nabla_S}\nabla_SY_{l,m}(\theta,\varphi) \Big],
\end{align}
\subsection{Decoupled System of Equations}
For notational convenience, we define the following operators:
$$
d_l^+=\frac{d}{dr}+\frac{l}{r},\quad d_l^-=\frac{d}{dr}-\frac{l}{r},
$$
$$\displaystyle \mathcal{L}_l = \frac{l(l+1)}{r^2}- \frac{d^2}{dr^2}-\frac{1}{r}\frac{d}{dr}.$$ 
The strong form of the decoupled  reduced differential equations is presented below. Detailed derivation of the strong form can be found in appendix C.

The system to solve $t_{i,l,m}(r)$ is:
 \begin{align}
&\alpha t_{i,l,m} + \beta \mathcal{L}_l (t_{i,l,m}) = f_{i,l,m}^{1,T} + \frac{f_{i,l,m}^{2,r}}{r} - \frac{1}{r} \frac{\partial (r f_{i,l,m}^{2,\nabla_S})}{\partial r},  \;\; {\rm in} \; \; I_i,\\
& t_{1,l,m}({\blue r_1})=t_{2,l,m}({\blue r_1}), \; \; t_{2,l,m}({\blue r_2})=t_{3,l,m}({\blue r_2}), \\
& \bar{\beta}_1 d_1^+ t_{1,l,m} ({\blue r_1}) - \bar{\beta}_2 d_1^+ t_{2,l,m} ({\blue r_1}) = g^{T}_{1,l,m},\\
& \bar{\beta}_3 d_1^+ t_{3,l,m} ({\blue r_2}) - \bar{\beta}_2 d_1^+ t_{2,l,m} ({\blue r_2}) = g^{T}_{2,l,m};\\
&  d_1^+t_{3,l,m} ({\blue r_3})=0,
\end{align}
And the system to solve $A_{i,l,m}(r)$ is:

\begin{align}
&\alpha ( r A_{i,l,m}(r))' + \bar\beta \Big( r \mathcal{L}A_{i,l,m}(r) \Big)'= r f_{i,l,m}^{1,\nabla_S}(r)
%\Big(r \frac{ rf_{i,l,m}^{1,r}(r)}{l+1}\Big)' 
+  (r f_{i,l,m}^{2,T}(r))', \\
&A_{1,l,m}({\blue r_1})=A_{2,l,m}({\blue r_1}), \; \; A_{2,l,m}({\blue r_2})=A_{3,l,m}({\blue r_2}),\\
&A_{1,l,m}'({\blue r_1})=A_{2,l,m}'({\blue r_1}), \; \; A_{2,l,m}'({\blue r_2})=A_{3,l,m}'({\blue r_2}),\\
&\bar{\beta}_2 \mathcal{L}(A_{2,l,m}({\blue r_1}))-\bar{\beta}_1 \mathcal{L}(A_{1,l,m}({\blue r_1}))=g^{\nabla_S}_{1,l,m},\\
&\bar{\beta}_2 \mathcal{L}(A_{2,l,m}({\blue r_2}))-\bar{\beta}_3 \mathcal{L}(A_{3,l,m}({\blue r_2}))=g^{\nabla_S}_{2,l,m},\\
&\mathcal{L}_l (A^3_{l,m})({\blue r_3}) = 0.
\end{align}

The solution space $X_N$ can be expanded from basis constructed from Legender polynomials:
\begin{align*}
&  \phi_k=(L_{k-1}-L_{k+1},0,0),\quad k=1,\dots,n-1\\
&\phi_{k+N-1}=(0,L_{k-1}-L_{k+1},0),\quad k=1,\dots,n-1\\
&\phi_{k+2N-2}=(0,0,L_{k-1}-L_{k+1}),\quad k=1,\dots,n-1\\
&\phi_{3N-2}=(-\frac{x}{2}+\frac{1}{2},0,0),\quad
\phi_{3N-1}=(\frac{x}{2}+\frac{1}{2},-\frac{x}{2}+\frac{1}{2},0),\\
&\phi_{3N}=(0,\frac{x}{2}+\frac{1}{2},-\frac{x}{2}++\frac{1}{2}),\quad
\phi_{3N+1}=(0,0,\frac{x}{2}+\frac{1}{2}).
\end{align*}
Notice this $\phi(x)$ has domain $x\in[-1,1]$ in each subdomain, we can convert it to the function by change of variable to $\phi(r)$, such that $r\in I_i$. In other words, 
$$
t_{l,m}^N(r)=\sum_{k=1}^{3N+1} u_k \phi_k(r),\quad A^N_{l,m}(r)=\sum_{k=1}^{3N+1} v_k\phi_k(i)
$$
We then plug back the expansions (\ref{ExpanNM}) into equation {\blue(\ref{*})}. Denote $(u,v)_\omega$ as the weighted integral over three domains $\ds \sum_{i=1}^3\int_{I_i}uv\omega dr$, and use $t_{l,m}, A_{l,m}$ as the piecewise function with function value $t_{i,l,m}, A_{i,l,m}$ respectively in $I_i$. Then the weak formulation of the reduced dimension system becomes: to find   $t_{l,m}^N(r), A_{l,m}^N(r)$, such that 
for $\phi(r)\in X_N$:
\begin{align}\label{wkt}
  \begin{split}
&\alpha(t_{l,m}^N,\phi)_{r^2}+(\bbt d_r t_{l,m}^N,d_r \phi)_{r^2} + l(l+1)(\bbt t_{l,m}^N,\phi) + {\blue r_1}(\bbt_1-\bbt_2)t_{l,m}^N({\blue r_1})\phi({\blue r_1}) \\
&+ {\blue r_2}(\bbt_2-\bbt_3)t_{l,m}^N({\blue r_2})\phi({\blue r_2}) 
+{\blue r_3}\bbt_3 t^N_{l,m}({\blue r_3})\phi({\blue r_3})\\
=&(\Pi f^{1,T}_{l,m}, \phi)_{r^2} + (\Pi f^{2,T}_{l,m} \phi)_r-(\Pi(d_r(rf^{2,\nabla_S}_{l,,m})),\phi)_r+ {\blue r_1}^2 g^{1,T}_{l,m}\phi^t_{l,m}({\blue r_1}) - {\blue r_2}^2 g^{2,T}_{l,m}\phi^t_{l,m}({\blue r_2}).
  \end{split}
\end{align}
\begin{align}\label{wkA}
\begin{split}
&\alpha(d_rA^N_{l,m},\phi)_{r^3} + \alpha (A^N_{l,m}\phi)_{r^2} + \bar\beta l(l+1) [(d_r A^N_{l,m},\phi)_r-(A^N_{l,m},\phi)] \\
&-\bar\beta [(d_rA^N_{l,m},\phi'')_{r^3}+2 (d_rA^N_{l,m},d_r \phi)_{r^2}]
+\bar\beta_3 {\blue r_3}^2(A^N_{l,m})'({\blue r_3})({\blue r_3}\phi'({\blue r_3})+2\phi({\blue r_3})) \\
&+(\bar\beta_2- \bar\beta_3)(A^N_{l,m})'({\blue r_2})({\blue r_2}\phi'({\blue r_2})+2\phi({\blue r_2}))+(\bar\beta_1-\bar\beta_2)(A^N_{l,m})'({\blue r_1})({\blue r_1}\phi'({\blue r_1})+2\phi({\blue r_1}))\\
&-{\blue r_3}l(l+1)\bar\beta_3 A^N_{l,m}({\blue r_3})\phi({\blue r_3}) 
+{\blue r_2}({\blue r_2}^2g_2^{\nabla_S}-l(l+1)(\bar\beta_2 -\bar\beta_3) A^N_{l,m}({\blue r_2}) ) \phi({\blue r_2})\\
&+{\blue r_1}(l(l+1)(\bar\beta_2-\bar\beta_1)A^N_{l,m}({\blue r_1})-{\blue r_1}^2g_1^{\nabla_S})\phi({\blue r_1})\\
=&( (I_Nf)', \phi )_{r^3} + ( I_Nf ,\phi )_{r^2}.
\end{split}
\end{align}

{\blue
Although we only discussed a first-order time marching scheme for brevity, it is clear that a similar second-order scheme based on backward difference formula and Adam-Bashforth extrapolation for nonlinear terms can be constructed, and it is expected that similar stability result can also be established.  
}

\section{Numerical Results}

Now we perform some numerical tests to validate our code. 

We consider an application to a solar interface dynamo  as in \cite{Chan06} where a finite-element method is used.
 The domain $\Omega$, composed of inner core $\Omega_1$, convection zone $\Omega_2$,
and exterior region $\Omega_3$, with the interfaces at $r_1=1.5$, $r_2 = 2.5$, $r_3=7.5$.
The magnetic diffusivity $\beta_i(x)$ is a  constant in each zone, namely $\{1,1,150\}$.
In the convection zone, the tachocline is located at $r_t=1.875$.
We set 
\begin{align}
f(x,t)= \sin^2\theta \cos \theta \sin \bigr[ \pi \frac{r-r_t}{r_2-r_t}\bigr],
\end{align}
which represents alpha quenching lies in between the tachocline and outer surface of convection zone; and
 take 
\begin{align}
&{\bf u} = (0,0,\Omega_t(\theta)r\sin\theta \sin \bigr[ \pi\frac{r-r_1}{r_t-r_1}\bigr]),\\
&\Omega_t(\theta) = 1 - 0.1642 \cos^2 \theta - 0.1591 \cos^4 \theta,
\end{align}
which represents a solar-like internal differential rotation in between the tachocline and the inner surface of convection zone.

The initial condition is given by 
\begin{align}
&{\bf B}_r = 2\cos \theta r(r-r_2)^2/r_2^2,\\
&{\bf B}_{\theta} = - \sin \theta (3r(r-r_2)^2 + 2r^2(r-r_2))/r_2^2,\\
&{\bf B}_{\varphi} = 3 \cos \theta \sin \theta r^2 (r-r_2)^2/r_2^2,
\end{align}
which is non-zero  only in the inner core and convection zone.

\begin{figure}[h]
\centering
\includegraphics[scale=.9]{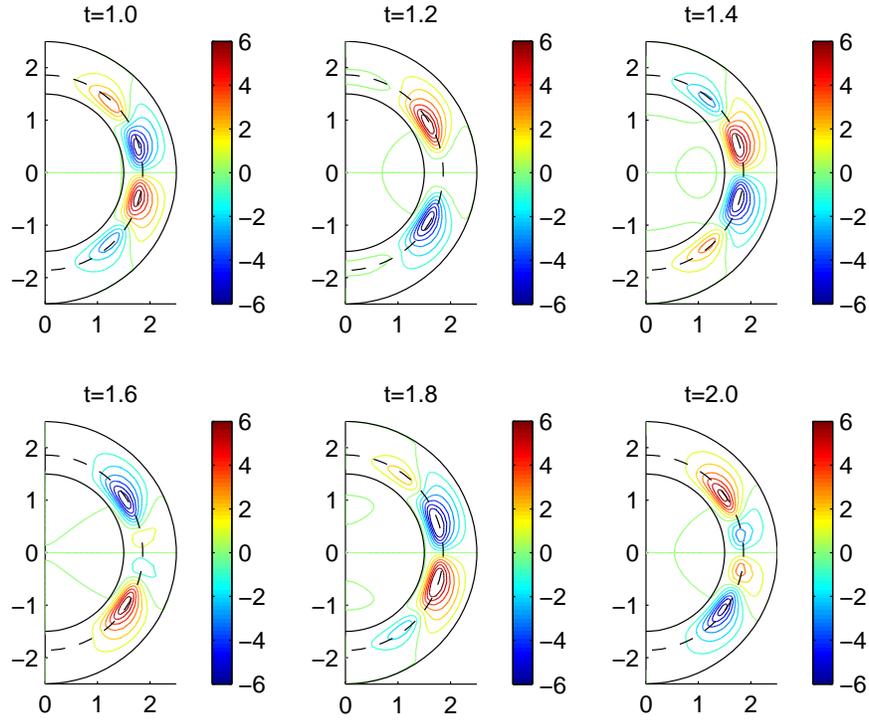}
\caption{Rm=100. Contours of the azimuthal field $B_{\varphi}$ in a meridional plane at different time.}
\label{contour}
\end{figure}

In the first simulation, we take $R_{\alpha} = 30$, $R_m = 100$, and plot in Figure \ref{contour}  the contours of azimuthal field $B_{\varphi}$ in a meridional plane. In this simulation, we take  $\delta t = \frac{1}{16000}$ with 40 equal spaced points for latitude and 40 equal spaced points for longitude, and 20 Legendre-Gaussian-Lobatto points in each layer. In Figure \ref{but}, we show the butterfly-shaped profile on the tachocline, where the function $f$ and internal differential rotation $u$ meet.
\begin{figure}[htp]
\centering
\includegraphics[scale=.5]{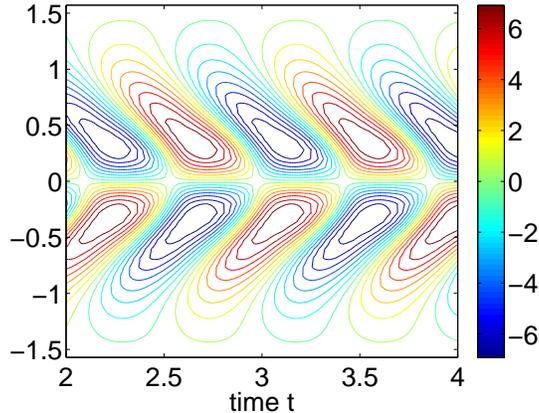}
\caption{Butterfly diagram of azimuthal field at the interface at tachocline}
\label{but}
\end{figure}

In the second simulation, we keep $R_{\alpha} = 30$ but take $R_m=1000$. and plot the contours of azimuthal field $B_{\varphi}$ in Figure \ref{contour2}. We observe similar quasi-periodic patterns as the previous example. The large $R_m$ leads to a significant increase of the magnitude.% Accordingly, a smaller time step, $\delta t=??$, has to be used.

In Figure \ref{energy}, we plot magnetic energy $E_m = \int_{\Omega} |B|^2 dx$ for $R_{\alpha}=30$ with different $R_m$. These results are consistent with those reported in \cite{Chan06}.

\begin{figure}[htp]
\includegraphics[scale=.9]{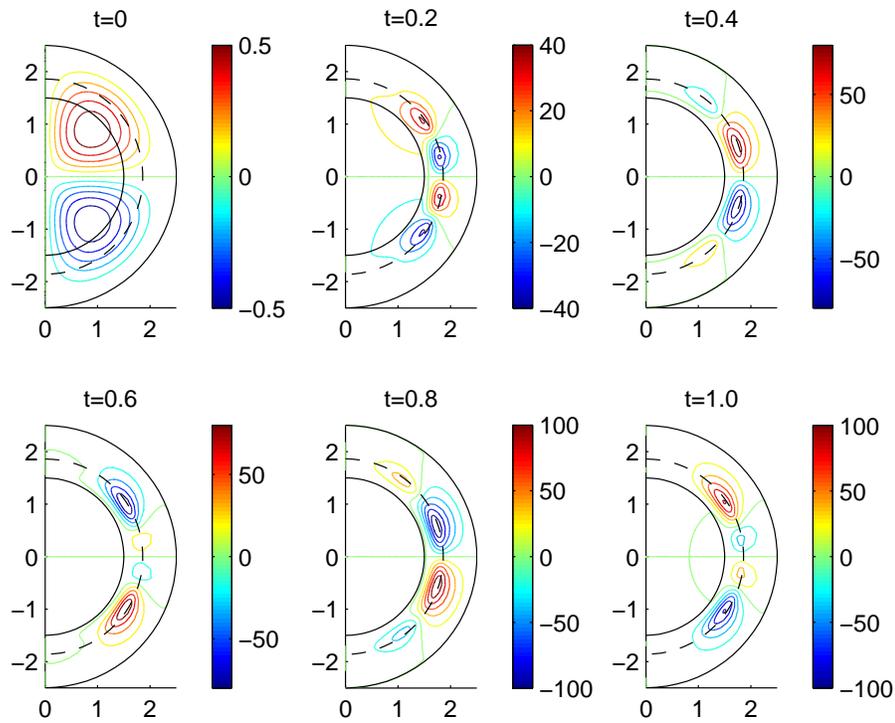}
\caption{Rm=1000. Contours of the azimuthal field $B_{\varphi}$ in a meridional plane at different time.}
\label{contour2}
\end{figure}

\begin{figure}[htp]
\centering
\includegraphics[scale=.5]{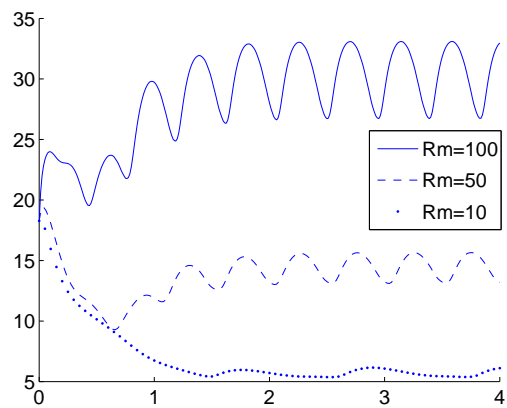}
\caption{Energy for $R_{\alpha}=30, R_m=10, 50, 100$.}
\label{energy}
\end{figure}

%{\color{red}
%Results evaluating divergence of each domain are attached below in figure (\ref{div}). The first 3 figures are all for Rm=50, number of time steps are all 20000, stopping time is t=4,2,1 respectively, so by decreasing the dt, the result is quite consistent. The 4th graph is for Rm=100. 
%\begin{figure}\label{div}
%\includegraphics[scale=.4]{div50}
%\includegraphics[scale=.4]{div50t2}\\
%\includegraphics[scale=.4]{div50t1}
%\includegraphics[scale=.4]{div100}
%\end{figure}
%
%}

\section{Concluding remarks}
We developed in this paper an efficient numerical scheme for the 3D mean-field spherical dynamo equation. For the time discretization, we adopt a special  semi-implicit discretization in such a way that at each time step one only needs to solve a linear system  with piecewise constant coefficients. To deal with the divergence-free constraint, we use the divergence free vector spherical harmonic functions in space so that our numerical solution is automatically divergence-free. In addition, this allows us to reduce the linear system to be solved at each time step to a  sequence of one-dimensional equations in the radial direction, which can then be  solved by using a  spectral-element method. Hence, the overall scheme is very efficient and accurate.

  We showed  that the solution of  our fully discretized scheme remains bounded independent of the number of unknowns, and presented several 
numerical results to validate our scheme.

\section{Acknowlegdegement} 

The research of T. Cheng is supported by NSF of China DOS 11871240  and DOS 11771170. The research of L. Ma is partially supported by NSF DMS-1913229. The research of J. Shen  is partially supported by NSF DMS-1620262,  DMS-1720442 and AFOSR FA9550-16-1-0102.
\appendix

%Important formula using vector spherical harmonic functions are presented in this section. 

\section{Vector Spherical Harmonic basis}
For VSH defined in (\ref{vsh})
\begin{align*}\label{vectcurl}
&\nabla \times \big(f \sphv\big)=\big(d_{l+2}^+ f\big)\spht,\\
&\nabla \times \big(f \sphw\big)=-\big(d_{l-1}^- f\big)\spht,\\
&\nabla \times (f \spht)=\frac{l(l+1) f}{r} Y_l^m {\bs e}_r+ \frac{1}{r}\frac{d(rf)}{d r} \nabla_S Y_l^m\\
& (2l+1)\nabla \times \big(f \spht\big)=(l+1)\big(d_{l+1}^+ f\big)\sphw -l\big(d_{l}^- f\big)\sphv,\\
&\nabla\times\nabla\times(f(r)\spht) =\left( \frac{l(l+1)f}{r^2} - \frac{2f'}{r} -{f''}\right) \spht = \mathcal{L}(f)\spht.
\end{align*}
\section{Representation of the Solenoidal Vector Field}
We now seek the representation of the divergence free space, in other words, the Solenoidal vector field. We know that $\nabla \cdot {\bf curl}_S Y_l^m = 0$, only need to see the other two sets.

Suppose we have a vector $\bs u$ represented by both sets of basis: 
\begin{align}%\label{abvw}
\begin{split}
\bs u = &\sum_{l,m} a_l^m(r) Y_l^m {\bs e}_r + b_l^m(r) \nabla_S Y_l^m\\
= & \sum_{l,m} v_l^m(r) \sphv + w_l^m(r) \sphw
\end{split}
\end{align}

Given identities
\begin{align}
&\nabla \cdot (f \spht) = 0, \quad
\nabla \cdot (f \sphv) = (l+1) d_{l+2}^{+} f Y_l^m,\\
&\nabla \cdot (f \sphw) = l d_{l-1}^{-} f Y_l^m,
\end{align}
Divergence of the vector $\bs u$ given expansion under basis $\spht, \sphv, \sphw$ is
\begin{align}
\nabla \cdot \bs u = \sum_{l,m} [(l+1)d_{l+2}^{+} v_l^m + l d_{l-1}^{-} w_l^m] Y_l^m
\end{align}

We know that if $u$ is divergence free, it must obey the following relation, 
\begin{align}
(l+1)d_{l+2}^{+} v_l^m + l d_{l-1}^{-} w_l^m = 0, \quad \forall l,m>0,
\end{align}
for $l=0$, we only have $\bs V_0^0$, therefore
\begin{align}
d_2^+(a_0^0(r)) = 0.
\end{align}

Now we take divergence on $\bs u$.
\begin{align}
\begin{split}
\nabla \cdot \bs u =& \sum_{l,m} \frac{1}{r^2} \frac{\partial(r^2 a_l^m)}{\partial r} Y_l^m+ \frac{1}{r\sin\theta} \frac{\partial}{\partial\theta} \left( \sin\theta b_l^m(r) \frac{\partial Y_l^m}{\partial\theta}\right) + \frac{1}{r\sin\theta} \frac{\partial}{\partial\varphi} \left( \frac{b_l^m(r)}{\sin\theta} \frac{\partial Y_l^m}{\partial \varphi} \right)\\
=& \sum_{l,m} \frac{1}{r^2} \frac{\partial(r^2 a_l^m)}{\partial r} Y_l^m+ \frac{b_l^m(r)}{r\sin\theta} \frac{\partial}{\partial \theta} \left( \sin\theta \frac{\partial Y_l^m}{\partial \theta} \right) + \frac{b_l^m(r)}{r\sin^2\theta} \frac{\partial^2Y_l^m}{\partial\varphi^2}\\
=&\sum_{l,m} \frac{1}{r^2} \frac{\partial(r^2 a_l^m)}{\partial r} Y_l^m+ \frac{b_l^m(r)}{r} \Delta_S Y_l^m
= \sum_{l,m} \frac{1}{r^2} \frac{\partial(r^2 a_l^m)}{\partial r} Y_l^m - l(l+1)\frac{b_l^m(r)}{r} Y_l^m
\end{split}
\end{align}

So we know for the solenoidal field, we need to have, $\forall l,m$,
\begin{align}
r (a_l^m)' + 2 a_l^m - l(l+1) b_l^m = 0
\end{align}
Consider the relations between $\{a_l^m, b_l^m\}$ and $\{v_l^m, w_l^m\}$ in (\ref{vectcurl})
\begin{align}
(l+1) v_l^m + lw_l^m = a_l^m,\quad w_l^m-v_l^m = b_l^m\\
v_l^m=\frac{a_l^m-lb_l^m}{2l+1},\quad w_l^m = \frac{a_l^m + (l+1)b_l^m}{2l+1}
\end{align}

So the coefficient for $\sphv, \sphw$ should be:
\begin{align}
v_l^m =- \frac{r}{(2l+1)(l+1)} d_{l-1}^- (a_l^m), \quad w_l^m = \frac{r}{l(2l+1)} d_{l+2}^+ (a_l^m)
\end{align}
or,
\begin{align}
v_l^m =- \frac{1}{(2l+1)(l+1)} d_{l}^- (ra_l^m), \quad w_l^m = \frac{1}{l(2l+1)} d_{l+1}^+ (ra_l^m)
\end{align}
Let $A_l^m(r) = \frac{r}{l(l+1)}a_l^m(r)$,
and notice the identities:
\begin{equation}\label{vectcurl}
\begin{split}
&{\bf curl}\big(f \sphv\big)=\big(d_{l+2}^+ f\big)\spht,\quad
{\bf curl}\big(f \sphw\big)=-\big(d_{l-1}^- f\big)\spht,\\
& (2l+1){\bf curl}\big(f \spht\big)=(l+1)\big(d_{l+1}^+ f\big)\sphw -l\big(d_{l}^- f\big)\sphv.
\end{split}
\end{equation}

Therefore we can rewrite $\bs u$ as:
\begin{align}
\bs u = {\bf curl} (\sum_{l,m}A_l^m(r) \spht)
\end{align}

%Noted here, $\bs u$ can also be represented as 
%\begin{align}
%\bs u = {\bf curl} (\sum_{l,m}A_l^m(r) \spht + \nabla f)
%\end{align}
%for any smooth function $f$. 
%%Lets expand $f=\sum_{l,m} f_l^m(r) Y_l^m$, then \[\nabla f= \sum_{l,m} (f_l^m)'Y_l^m {\bs e}_r + \frac{f_l^m}{r}\nabla_SY_l^m\]
%%Take $\bf curl$, and we get,
%%\begin{align}
%%{\bf curl} \nabla f = \sum_{l,m} \frac{f'}{r} \spht + 0 {\bs e}_r + \frac{f'}{r} (-\frac{1}{\sin\theta} \frac{\partial Y_l^m}{\partial \varphi} {\bs e}_{\theta} + \frac{\partial Y_l^m}{\partial \theta} {\bs e}_{\varphi})=\bs 0
%%\end{align}
%%
%However, this will non uniqueness will not affect the result, since $\spht$ is orthogonal to the space spanned by $\{{\bs e}_r, \nabla_S Y_l^m\}$, so the representation of $A_l^m(r)$ is still unique.

Now we know for any $\bs u$ in a solenoidal field, we can expand it as:
\begin{align}
\bs u = \sum_{l,m} t_l^m(r) \spht + \nabla \times (A_l^m(r) \spht) + a_0^0(r) Y_0^0 {\bs e}_r
\end{align}
with $\displaystyle d_2^{+} a_0^0(r) = 0$. For most practical cases, $a_0^0(r)$ is zero. 

\section{Derivation of the Strong Form for Solenoidal Vector Field}
We will give detailed derivation of the strong form in the solenoidal expansion in this section. The system is: 
\begin{align}
\alpha \bs B^n + \bar\beta\nabla\times \nabla\times B^n = \bs f_1+\nabla \times \bs f_2, \quad {\rm in }\;\; \Omega,
\end{align}
with the boundary conditions
\begin{equation}
\begin{split}
   & \nabla\times\BB^n\times \nn=0, \quad  \mbox{on}  \quad {\pa \Omega},\\
& [\bar \be\cu\BB^n\times \nn ]= \bs g, \ [\BB^n ]=0
     \ \ \ \ \mbox{on} \;\; \Gamma_1 \cup \Gamma_2.
\end{split}
\end{equation}
The expansions for functions involved are:
\begin{align}
&\bs B_{\bs N} (r,\theta,\varphi) = \sum_{i=1}^3\sum_{l=0}^\infty \sum_{|m|= 0}^{l} u_{I_i}\Big[  t_{l,m}^i(r) \spht(\theta,\varphi) +  \nabla\times \Big( A_{l,m}^i(r) \spht(\theta,\varphi) \Big) \Big],
%&\bs A_{\bs N} (r, \theta, \varphi) = \sum_{l=0}^M \sum_{|m|=0}^{l} \Big [ \phi_{l,m}^t(r) \spht(\theta,\varphi) + \nabla \times \Big(\phi_{l,m}^A(r) \spht(\theta,\varphi) \Big)\Big],
\end{align}

\begin{align}
&\bs f_1(r,\theta,\varphi) = \sum_{i=1}^3\sum_{l=0}^\infty \sum_{|m|= 0}^{l} u_{I_i}\Big[  f_{i,l,m}^{1,T} (r)\spht(\theta,\varphi) + f_{i,l,m}^{1,\nabla_S}(r) \nabla_S Y_l^m(\theta,\varphi) + f_{i,l,m}^{1,r} (r){\bs e}_r\Big],\\
&\bs f_2(r,\theta,\varphi) = \sum_{i=1}^3\sum_{l=0}^\infty \sum_{|m|= 0}^{l} u_{I_i}\Big[  f_{i,l,m}^{2,T} (r)\spht(\theta,\varphi) + f_{i,l,m}^{2,\nabla_S} (r)\nabla_S Y_l^m (\theta,\varphi)+ f_{i,l,m}^{2,r} (r){\bs e}_r\Big].\\
&\bs g_i(\theta,\varphi)= \sum_{l=0}^\infty \sum_{|m|= 0}^{l} \Big[  g_{i,l,m}^{T}(r) \spht(\theta,\varphi) + g_{i,l,m}^{\nabla_S}\nabla_SY_{l,m}(\theta,\varphi) \Big],
\end{align}
After applying double curl on $\bs B_N$, 
\begin{align}
\nabla\times (\nabla \times \bs B_N) = \sum_{i=1}^3\sum_{l=0}^\infty \sum_{|m|= 0}^{l} u_{I_i}\Big[\mathcal{L}(t_{i,l,m}(r)) \spht + \frac{l(l+1) \mathcal{L} (A_{i,l,m}(r)) }{r} Y_l^m {\bs e}_r \\
+ \frac{1}{r} \frac{\partial \Big(r \mathcal{L} (A_l^m(r)) \Big)}{\partial r} \nabla_S Y_l^m\Big].
\end{align}
Direct calculation on $\nabla \times {\bs f}_2$ gives,
\begin{align}
\nabla \times {\bs f}_2 =  \sum_{i=1}^3\sum_{l=0}^\infty \sum_{|m|= 0}^{l} u_{I_i}\Big[ \frac{l(l+1) f_{i,l,m}^{2,T}(r)}{r} Y_l^m (\theta,\varphi){\bs e}_r + \frac{1}{r}\frac{d(r f_{i,l,m}^{2,T}(r))}{d r} \nabla_S Y_l^m (\theta,\varphi)\\
- \frac{1}{r} \frac{d (r f_{i,l,m}^{2,\nabla_S}(r))}{d r} \spht(\theta,\varphi) + \frac{f_{i,l,m}^{2,r}(r)}{r} \spht(\theta,\varphi)\Big].
\end{align}
Due to the othorganality of $\spht, Y_l^m \bs e_r$ and $\nabla_S Y_l^m$, it is easy to derive that
for $\spht$ direction,
\begin{align}
\alpha t_{i,l,m}(r) +\bar \beta_i \mathcal{L} (t_{i,l,m}(r)) = f_{i,l,m}^{1,T} (r)+ \frac{f_{i,l,m}^{2,r}(r)}{r} - \frac{1}{r} \frac{d (r f_{i,l,m}^{2,\nabla_S}(r))}{d r},  \;\; {\rm in} \; \; I_i,
\end{align}
for $\nabla_S Y_l^m$ direction,
\begin{align}
\alpha \frac{1}{r} \frac{d ( r A_{i,l,m}(r)) }{d r} + \bar\beta \frac{1}{r} \frac{d \Big( r \mathcal{L}A_{i,l,m}(r) \Big)}{d r} = f_{i,l,m}^{1,\nabla_S}(r) + \frac{1}{r} \frac{d (r f_{i,l,m}^{2,T}(r))}{d r}, \label{dyn_ny}
\end{align}
for ${\bs e}_r$ direction,
\begin{align}
\alpha A_{i,l,m}(r) + \bar \beta_i \mathcal{L}(A_{i,l,m}(r))= \frac{r}{l(l+1)} f_{i,l,m}^{1,r}(r) + f_{i,l,m}^{2,T}(r).  \;\; {\rm in} \; \; I_i.\label{dyn_er}
\end{align}
Notice the fact that $\bs f_1$ is also in the solenoidal vector field, which means
\begin{align}
\frac{1}{r^2} \frac{d(r^2 f_{l,m}^{1,r})}{d r} = \frac{l(l+1)}{r} f_{l,m}^{1,\nabla_S},
\end{align}
(\ref{dyn_ny}) can be rewritten as:
\begin{align}
\alpha \frac{1}{r} \frac{d( r A_{i,l,m}(r)) }{d r} + \bar\beta \frac{1}{r} \frac{d\Big( r \mathcal{L}A_{i,l,m}(r) \Big)}{dr} = \frac{1}{r} \frac{d}{dr} (\frac{ r^2f_{i,l,m}^{1,r}(r)}{l+1}) + \frac{1}{r} \frac{d (r f_{i,l,m}^{2,T}(r))}{d r}, \label{dyn_ny1}
\end{align}
We want to make a remark that (\ref{dyn_er}) and (\ref{dyn_ny1}) differ in the order of the PDE in the sense that solution of $A_{i,l,m}(r)$ can differ up to a constant. We have already pointed out that in the solenoidal representation, the $A(r)$ is not unique, but we can manually set the constant to any value for convenience. 

We focus now on the boundary conditions. First we consider
$$
[\BB_N]=0,
$$
this requires the continuity at intersections, which leads to the following conditions. 
For $\spht$ direction:
 \begin{align}
& t_{1,l,m}(a)=t_{2,l,m}(a), \; \; t_{2,l,m}(b)=t_{3,l,m}(b), 
\end{align}
%& \bar{\beta}_3 t^3_{l,m} (c)= g^{3,T}_{l,m},\\
%& \bar{\beta}_1 d_1^+ t^1_{l,m} (a) - \bar{\beta}_2 d_1^+ t^2_{l,m} (a) = g^{1,T}_{l,m},\\
%& \bar{\beta}_3 d_1^+ t^3_{l,m} (a) - \bar{\beta}_2 d_1^+ t^2_{l,m} (a) = g^{3,T}_{l,m};
%\end{align}
%\begin{align}
%&A^{3}_{l,m} =\frac{1}{\alpha}(\frac{r}{l(l+1)} f_{l,m}^{1,r} + f_{l,m}^{2,T} + g_{l,m}^{3,\nabla_S})\\
for $\nabla_SY_l^m$ direction:
\begin{align}
\frac{d}{d r}(rA_{1,l,m}(r))\Big |_{r=a}=\frac{d}{d r}(rA_{2,l,m}(r))\Big |_{r=a},\;\;
\frac{d}{dr}(rA_{2,l,m}(r))\Big |_{r=b}=\frac{d}{d r}(rA_{3,l,m}(r))\Big |_{r=b}
\end{align}
for $Y_l^m \bs e_r$ direction:
\begin{align}
A_{1,l,m}(a)=A_{2,l,m}(a), \; \; A_{2,l,m}(b)=A_{3,l,m}(b).
\end{align}
Next, we consider 
$$
[\nabla \times \BB_N \times \bs n] = \bs g,
$$
which will leads to for $\spht$ direction:
\begin{align}
%&  d_1^+t^3_{l,m} (c)=0,\\
& \bar{\beta}_1 d_1^+ t^1_{l,m} (a) - \bar{\beta}_2 d_1^+ t^2_{l,m} (a) = g^{T}_{1,l,m},\\
& \bar{\beta}_3 d_1^+ t^3_{l,m} (b) - \bar{\beta}_2 d_1^+ t^2_{l,m} (b) = g^{T}_{2,l,m},
\end{align}
for $\nabla_S Y_l^m$ direction:
\begin{align}
&\bar{\beta}_2 \mathcal{L}(A_{2,l,m}(a))-\bar{\beta}_1 \mathcal{L}(A_{1,l,m}(a))=g^{\nabla_S}_{1,l,m},\\
&\bar{\beta}_2 \mathcal{L}(A_{2,l,m}(b))-\bar{\beta}_3 \mathcal{L}(A_{3,l,m}(b))=g^{\nabla_S}_{2,l,m},
\end{align}
and no condition can be given in the $\bs e_r$ direction.\\
On $\Gamma_3$, the boundary condition is:
$$
\nabla \times \BB_N \times \bs n = 0,
$$
this leads to
 \begin{align}
& d_1^+ t^3_{l,m} (c)= 0,\quad \mathcal{L}_l (A^3_{l,m})(c) = 0.
%& \bar{\beta}_1 d_1^+ t^1_{l,m} (a) - \bar{\beta}_2 d_1^+ t^2_{l,m} (a) = g^{1,T}_{l,m},\\
%& \bar{\beta}_3 d_1^+ t^3_{l,m} (a) - \bar{\beta}_2 d_1^+ t^2_{l,m} (a) = g^{3,T}_{l,m};
%\end{align}
%\begin{align}
%&A^{3}_{l,m} =\frac{1}{\alpha}(\frac{r}{l(l+1)} f_{l,m}^{1,r} + f_{l,m}^{2,T} + g_{l,m}^{3,\nabla_S})\\
\end{align}
It is quite clear the for $\nabla_S Y_l^m$ and $Y_l^m \bs e_r$ directions, the differential equations are essentially the same but boundary conditions differ a lot. This is because the $\bs e_r$ direction is a consequence in the solenoidal vector field. We will take the $\nabla_S Y_l^m$ as the first choice, and still taking account the boundary conditions for $\bs e_r$ direction. 

Now we can summarize the strong form for $t_{i,l,m}(r)$ and $A_{i,l,m}(r)$.
\begin{align}
&\alpha t_{i,l,m}(r) +\bar \beta_i \mathcal{L} (t_{i,l,m}(r)) = f_{i,l,m}^{1,T} (r)+ \frac{f_{i,l,m}^{2,r}(r)}{r} - \frac{1}{r} \frac{d (r f_{i,l,m}^{2,\nabla_S}(r))}{d r},  \;\; {\rm in} \; \; I_i,\\
& t_{1,l,m}(a)=t_{2,l,m}(a), \; \; t_{2,l,m}(b)=t_{3,l,m}(b),\\
& \bar{\beta}_1 d_1^+ t^1_{l,m} (a) - \bar{\beta}_2 d_1^+ t^2_{l,m} (a) = g^{T}_{1,l,m},\\
& \bar{\beta}_3 d_1^+ t^3_{l,m} (b) - \bar{\beta}_2 d_1^+ t^2_{l,m} (b) = g^{T}_{2,l,m},\\
& d_1^+ t^3_{l,m} (c)= 0.
\end{align}

\begin{align}
&\alpha ( r A_{i,l,m}(r))' + \bar\beta \Big( r \mathcal{L}A_{i,l,m}(r) \Big)'= r f_{i,l,m}^{1,\nabla_S}(r)+  (r f_{i,l,m}^{2,T}(r))', \\
&A_{1,l,m}(a)=A_{2,l,m}(a), \; \; A_{2,l,m}(b)=A_{3,l,m}(b),\\
&A_{1,l,m}'(a)=A_{2,l,m}'(a), \; \; A_{2,l,m}'(b)=A_{3,l,m}'(b),\\
&\bar{\beta}_2 \mathcal{L}(A_{2,l,m}(a))-\bar{\beta}_1 \mathcal{L}(A_{1,l,m}(a))=g^{\nabla_S}_{1,l,m},\\
&\bar{\beta}_2 \mathcal{L}(A_{2,l,m}(b))-\bar{\beta}_3 \mathcal{L}(A_{3,l,m}(b))=g^{\nabla_S}_{2,l,m},\\
&\mathcal{L}_l (A^3_{l,m})(c) = 0.
\end{align}

%for  continuity of $\beta(x) \nabla \times\bs B_N\times \bs n$ at intersections,
%we observe that for $\spht$ direction,
%\begin{align*}
%\beta_1(a) d_1^+ t^1_{l,m}(a) = \beta_2(a) d_1^+ t^2_{l,m}(a), \;\; \beta_3(b) d_1^+ t^3_{l,m}(b) = \beta_2(b) d_1^+ t^2_{l,m}(b),
%\end{align*}
%similarly, we have a second order boundary condition for $\nabla_S Y_l^m$ direction
%\begin{align}
%\beta_1(a) \mathcal{L}_l A^{1}_{l,m} (a) = \beta_2(a) \mathcal{L}_l A^2_{l,m}(a),
%\;\;
%\beta_3(b) \mathcal{L}_l A^{3}_{l,m} (b) = \beta_2(b) \mathcal{L}_l A^2_{l,m}(b).
%\end{align}
%However, now we have more boundary conditions than needed.
%We now show that (6.11) is a set of unnecessary boundary conditions. 
%
%In the original model, $f(x,t)$ and $\bs u $ has compact support on $I_2$, therefore near intersections, terms in $\bf f_2$ in (2.14) are almost zero except 
%$(\bar{\beta}-\beta(x) )\nabla \times \bs B^{n-1}$. Written under expansion, $f^{2,T}_{l,m} $ corresponds to $(\bar{\beta} - \beta(x)) \mathcal{L}_l A^i_{l,m}$. 
%By continuity of previous tilmestep and first and zeroth derivative of $A^i_{l,m}$ at intersections, the above boundary conditions are automatically satisfied. 
\bigskip

\bibliographystyle{plain}

\bibliography{refpapers_dynamo,div}

\end{document}